\documentclass{article}
\usepackage{amsmath, amssymb, amsfonts, amsthm}
\usepackage{graphicx}
\usepackage{mathtools}
\usepackage{float}
\usepackage[utf8]{inputenc}

\usepackage{verbatim}
\usepackage{tikz}
\usepackage{enumerate}
\usepackage{todonotes}
\usepackage{caption}
\usepackage{xparse}
\usepackage{standalone}
\usepackage{enumitem}
\usepackage[backend=biber, sorting=none]{biblatex}
\usepackage{hyperref}

\hypersetup{
    colorlinks=true,
    allcolors=blue 
}
\usepackage{biblatex} 
\usepackage{cleveref}
\usepackage[linesnumbered, boxed, lined, ruled,vlined]{algorithm2e}
\usepackage{enumerate}
\usepackage{todonotes}
\usepackage{longtable}

\newtheorem{theorem}{Theorem}[section]
\newtheorem{corollary}[theorem]{Corollary}
\newtheorem{lemma}[theorem]{Lemma}

\theoremstyle{definition}

\newtheorem{innercustomlemma}{Lemma}
\newenvironment{customlemma}[1]
  {\renewcommand\theinnercustomlemma{#1}\innercustomlemma}
  {\endinnercustomlemma}

\newtheorem{innercustomthm}{Theorem}
\newenvironment{customthm}[1]
  {\renewcommand\theinnercustomthm{#1}\innercustomthm}
  {\endinnercustomthm}

\newcommand{\setJ}{P}
\newcommand{\numJ}{N}

\newcommand{\sched}{S}

\newcommand{\block}{B}
\newcommand{\job}{p}
\newcommand{\temp}{T}
\newcommand{\tempJ}{t}
\newcommand{\col}{C}
\newcommand{\colJ}{c}

\newcommand{\ie}{i.\,e.\@\xspace}

\newcommand{\wrt}{w.r.t.\@\xspace}
\newcommand{\wlg}{w.l.o.g.\@\xspace}

\NewDocumentCommand{\TminB}{oo}{%
  t_{\min}(\block
  \IfValueT{#1}{_{#1}}%
  \IfValueT{#2}{^{#2}}%
  )
}
\NewDocumentCommand{\TmaxB}{oo}{%
  t_{\max}(\block
  \IfValueT{#1}{_{#1}}%
  \IfValueT{#2}{^{#2}}%
  )
}

\title{Computing an optimal single machine schedule with sequence dependent setup times using shortest path computations}
\author{Dominik Leib\textsuperscript{1} \and Till Heller\textsuperscript{2} \and Raphael Kühn\textsuperscript{1}}
\date{\textsuperscript{1}Fraunhofer ITWM, 67663 Kaiserslautern, Germany\\
\phantom{0}\textsuperscript{2}Co-Fox GmbH, 29227 Celle, Germany}

\begin{document}

\maketitle

\begin{abstract}
    %We study a single-machine scheduling problem with sequence dependent setup times, which arises in various manufacturing and service applications, including flooring production. In particular, we focus on the calendering stage, where setup times are influenced by temperature and color transitions between jobs, significantly affecting throughput and energy efficiency. We propose a novel solution framework that transforms the scheduling problem into a path-finding problem on a specially constructed graph. By introducing a remodeling step that encodes sequence-dependent setups into the graph structure, we enable the use of classical shortest-path algorithms to identify optimal job sequences. Our method uncovers key structural properties of feasible optimal schedules and offers a practical and theoretically grounded optimization technique.

We study a single-machine scheduling problem with sequence dependent setup times, motivated by applications in manufacturing and service industries - in particular, the calendering stage in flooring production. In this phase, setup times are primarily driven by temperature and color transitions between consecutive jobs, with significant impact on throughput and energy efficiency. We present a novel solution framework that transforms the scheduling problem into a path-finding problem on a specially constructed layered graph. By encoding sequence-dependent effects directly into the graph’s structure, we enable the use of classical shortest-path algorithms to compute optimal job sequences. The resulting method is polynomial-time solvable for the two-color case and reveals key structural properties of optimal schedules. Our approach thus provides both a theoretically grounded and practically applicable optimization technique.
\end{abstract}

\section{Introduction} \label{sec:introduction}
    Single-machine scheduling with sequence-dependent setup times is a fundamental problem in combinatorial optimization. Here, a single resource must execute a sequence of jobs, where each transition between consecutive jobs incurs a setup cost depending on the ordered job pair (cf. \cite{santos1997scheduling, wang2011single}). %The general form, denoted \(1|s_{ij}|C_{\max}\) in standard scheduling notation, is strongly NP-hard.
This problem arises in many industrial settings—from semiconductor assembly and printing to chemical batch production—where changeovers critically affect throughput and resource utilization (cf. \cite{lin2022single, zheng2024single}). This computational challenge has motivated a wide range of exact, heuristic, and hybrid solution methods. Exact approaches, such as mixed integer linear programming (MILP) (cf. \cite{lin2022single, nesello2018exact}) and dynamic programming (cf. \cite{nesello2018exact}), deliver optimal solutions for small to medium instances. For larger problems, metaheuristics including iterated greedy (IG), variable neighborhood search (VNS), tabu search (TS), and genetic algorithms (GA) have proven effective (cf. \cite{balzereit2024scalable, deliktas2014single, lin2022single}). Recent developments also apply constraint programming (CP) and answer-set programming (ASP) to exploit logical modeling and global constraints (cf. \cite{heinz2022constraint}). The theoretical foundations are rich, with comprehensive surveys summarizing results across various machine models and objectives (cf. \cite{allahverdi1999review, yang1999survey, zhu2006scheduling}). Classical formulations often recast the problem as a traveling salesman problem (TSP) (cf. \cite{bagchi2006review, potts2009fifty}), while recent work explores bicriteria trade-offs and heuristic methods (cf. \cite{leib2022heuristic, schmitz2009bicriteria}).

In this work, we study a scheduling problem arising in the flooring production. The overall production involves several steps: first, in step 1 raw materials are blended in an industrial mixer to produce a homogeneous compound material; in step 2 this compound is processed in a calendering machine where it is rolled into uniform sheets under controlled heat and pressure. Precise temperature control during calendering is crucial to ensure surface quality and structural integrity. The sheets are subsequently cooled, optionally slit into narrower rolls in step 3, and finally pressed into their final shape in step 4 (see \Cref{fig:production}).

\begin{figure}[ht!]
    \centering
    \includegraphics[width=0.7\linewidth]{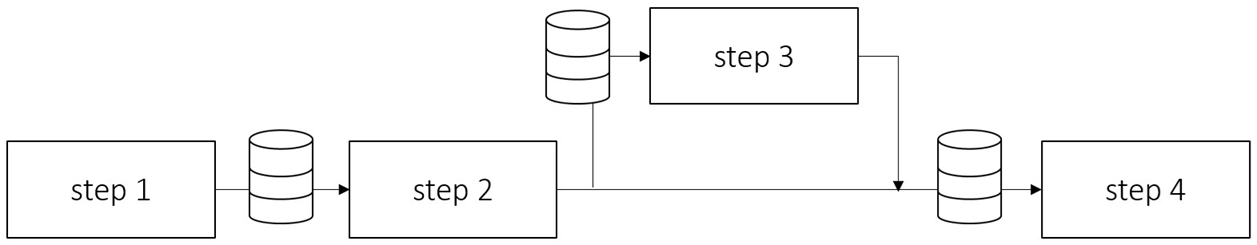}
    \caption{A schematic representation of the production process of flooring production.}
    \label{fig:production}
\end{figure}

We focus on the calendering step, where each job corresponds to a production batch characterized by a process temperature and a color. The sequence-dependent setup times stem from the energy-intensive temperature adjustments and cleaning efforts required when switching between different colors. Our model captures the setup cost as a combination of temperature differences and color changeovers. Although related to classical scheduling with sequence-dependent setup times, the specific structure of setup costs and the industrial context necessitate a specialized study.

This paper contributes three main results. First, we provide a formal problem definition and identify structural properties of optimal schedules, especially for two color groups. Second, leveraging these properties, we develop a graph representation in which feasible schedules correspond to paths. Third, we propose a polynomial-time shortest-path algorithm to find an optimal schedule under these constraints. The paper is organized as follows: \Cref{sec:problem} formally defines the problem; \Cref{sec:theoreticalresults} presents key theoretical insights; Section~\ref{sec:algorithm} details the graph-based algorithm and analyzes its complexity; and an outlook in \Cref{sec:outlook} discusses potential generalizations and open research questions.

\section{Problem Description}\label{sec:problem}
    We consider a set $\setJ$ of $N$ products that need to be processed on a machine; in the following, we refer to them as \emph{jobs}. Each job $\job$ is associated with a pair $(\tempJ, \colJ)$, where $\tempJ \in (0,\infty)$ denotes its \emph{process temperature} and $\colJ \in \{0,1\}$ its \emph{color}. We also use the notation $\tempJ(\job)$ and $\colJ(\job)$ to refer to the process temperature and color of job $\job \in \setJ$, respectively. Given any sequence of jobs $(\job_1,\ldots,\job_n)$ from $\setJ$, we define the \emph{total number of color changes} as 

\[\col(\job_1,\ldots,\job_n) := \sum_{i=1}^{n-1}|\colJ(\job_{i+1})-\colJ(\job_{i})|\]

and analogously define the \emph{total process temperature change} as 
\[\temp(\job_{1},\ldots,\job_{n}) := \sum_{i=1}^{n-1}|\tempJ(\job_{i+1})-\tempJ(\job_{i})|.\] 

A sequence that contains each job in $\setJ$ exactly once is called a \emph{schedule}. If $\setJ$ contains two jobs with the same process temperature and color, we assume without loss of generality that they are scheduled consecutively. To avoid notational overhead, we thus assume that $\tempJ(\job) \neq \tempJ(\job')$ for any pair of jobs with $\colJ(\job) = \colJ(\job')$ - that is, for each color, there is at most one job for each process temperature value.

The two objective functions are naturally conflicting. For example, sorting the jobs in increasing order of $\tempJ$ yields a schedule that is optimal with respect to $\temp$, but the total number of color changes may then reach its maximum of $N-1$. Conversely, a schedule that is optimal with respect to $\col$ consists, without loss of generality, of all jobs of color 0 followed by all jobs of color 1, which may result in a significantly higher total process temperature change than a $\temp$-optimal schedule.
\Cref{fig:optimalschedules} shows two such optimal schedules and illustrates the difference in total process temperature change. We visualize schedules in a specific way: the $y$-coordinate of each job represents its process temperature $\tempJ(p_i)$, and the $x$-coordinate corresponds to the cumulative temperature $\temp(p_1,\ldots,p_i)$. In this representation, the total process temperature change of the schedule corresponds exactly to its length along the $x$-axis.

\begin{figure}[ht!]
    \centering
        \includegraphics[width=0.8\textwidth]{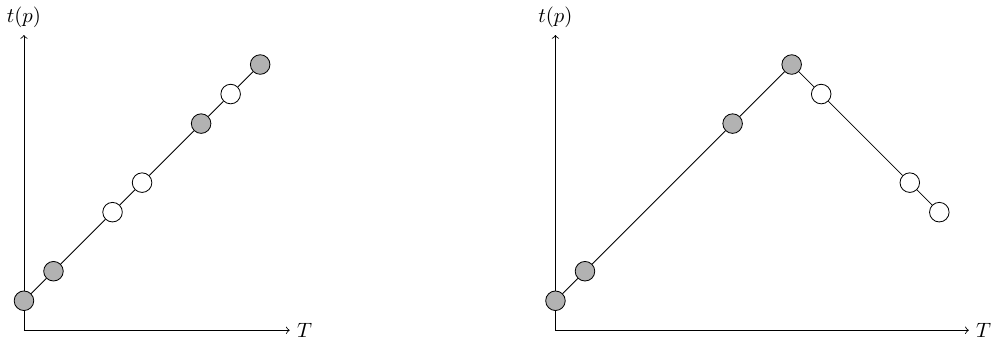}
        \captionsetup{width=0.8\textwidth}
    \caption{Two schedules for a set of jobs $P$, on the left a $\temp$-optimal schedule, on the right a $\col$-optimal schedule.}
    \label{fig:optimalschedules}
\end{figure}

The goal of the optimization problem is to find a schedule that minimizes the total process temperature change among all schedules with at most $\hat{k}$ total color changes for some given $\hat{k}$. Note that $\hat{k}$ is bound above by $\min\{\numJ_0, \numJ_1\}+2$, where $\numJ_c$ is the number of jobs of color $c$ in $\setJ$ and that for each $k \leq \hat{k}$ there exists at least one schedule $\col(\sched) = k$. We formalize the optimization problem as follows:

    \textbf{(P)} Given a set of jobs $\setJ$ and $\hat{k} \leq \min\{\numJ_0, \numJ_1\}+2$, find a schedule $\sched$ such that:
    \begin{itemize}
        \item The total number of color changes satisfies
        \[
        \col(\sched) = \sum_{i=1}^{N-1}|\colJ(\sched_{i+1})-\colJ(\sched_i)| \leq \hat{k}
        \]
        \item and the total process temperature change
        \[
        \temp(\sched) = \sum_{i=1}^{N-1}|\tempJ(\sched_{i+1}) - \tempJ(\sched_i)|
        \]
        is minimized, where $\sched_i$ denotes the job at position $i$ in $\sched$.
    \end{itemize}

The problem is closely related to the first known tractable variant of the Traveling Salesman Problem, the \emph{Gilmore-Gomory TSP} \cite{gilmore1964}, where nodes correspond to jobs with associated process temperatures, and edge lengths are given by the temperature differences of adjacent nodes. The main difference here is that the jobs correspond to pairs of process temperatures (i.e., the first and last jobs of a block), and scheduling allows flipping of these blocks. To our knowledge, the complexity of the problem with an arbitrary number of colors remains open.

In the following, we define a \emph{block} as a maximal sequence of consecutive jobs in $\sched$ that share the same color. Given a schedule $\sched$ with exactly $k$ color changes, it can be uniquely partitioned into $k+1$ such blocks, which we denote by $\block_1, \ldots, \block_{k+1}$.
We extend the notation $\temp(\cdot)$ naturally to blocks or sequences of blocks, so that $\temp(\block_i,\ldots,\block_j)$ denotes the total temperature difference when concatenating the jobs in blocks $\block_i$ through $\block_j$. For each block $B$ we define $\TmaxB := \max\{\tempJ(\job) \mid \job \in \block\}$ as the maximum process temperature within $B$, analogously for $\TminB$. We further define sorting operations on blocks, schedules, and job sequences as follows. The notation $\overrightarrow{B}$ denotes the block $B$ sorted in increasing order of process temperature, while $\overleftarrow{B}$ refers to $B$ sorted in decreasing order. Additionally, for any sequence of jobs, blocks, or an entire schedule, we use $\overline{(p_1, \ldots, p_n)} := (p_n, \ldots, p_1)$ to denote its reversal. For reference, a glossary of all notation used throughout the paper is provided at the end of the document.

\section{Theoretical Results} \label{sec:theoreticalresults}
    In this section, we provide theoretical results about optimal schedules to prepare for a polynomial-time algorithm. The idea is to gradually reduce the set of all possible schedules with each lemma to a subset that can be searched in polynomial time. The key observation is that there exists a schedule with minimal total temperature change and at most $k$ color changes that satisfies the following properties:

\begin{theorem}\label{thm:mainthm}
    There exists an optimal schedule \(\sched^* = (\block_1, \ldots, \block_{k+1})\) with \(k \leq \hat{k}\) such that the following properties hold:
    \begin{enumerate}[leftmargin=1.1cm, label=\alph*)]
        \item For each block \(\block_i = (\job_{i_1}, \ldots, \job_{i_{n_i}})\), either
        \[
        \tempJ(\job_{i_j}) < \tempJ(\job_{i_{j+1}}) \quad \text{for } j=1, \ldots, n_i - 1,
        \]
        or
        \[
        \tempJ(\job_{i_j}) > \tempJ(\job_{i_{j+1}}) \quad \text{for } j=1, \ldots, n_i - 1.
        \]
        We say the blocks are \emph{internally sorted} by process temperature. Additionally, each inner block \(\block_2, \ldots, \block_k\) is internally sorted in increasing order.
        
        \item For each pair of blocks \(\block_r, \block_s\) with \(r < s\) of the same color, it holds that
        \[
        \TmaxB[r] < \TminB[s],
        \]
        i.e., the two subsequences of same-colored blocks are \emph{externally sorted in increasing order}. In particular, two blocks of the same color do not \emph{intersect} with respect to the process temperatures of the jobs contained.
    \end{enumerate}
\end{theorem}

Reducing the problem to schedules with these properties allows us to define a search graph, polynomial in size with respect to $N$, such that an optimal schedule can be found via shortest path search. This establishes parameter tractability of the decision problem with the fixed parameter being the number of colors (which is 2). In the following, we call a schedule \(\sched'\) an \emph{improvement} of \(\sched\) if \(\temp(\sched') \leq \temp(\sched)\) and \(\col(\sched') \leq \col(\sched)\).

\begin{lemma}\label{lem:internalorderingblock}
~ 
\begin{enumerate}[leftmargin=1.1cm, label=\alph*)] 
    \item Any block \(\block\) of jobs satisfies \(\temp(\block) \geq \TmaxB - \TminB\), and if \(\block\) is internally sorted by process temperature, then equality holds.
    \item Any schedule \(\sched\) can be improved to \(\sched'\) such that all blocks of \(\sched'\) are internally sorted.
\end{enumerate}
\end{lemma}

\begin{figure}[ht!]
    \begin{center}
        \includegraphics[width=0.8\textwidth]{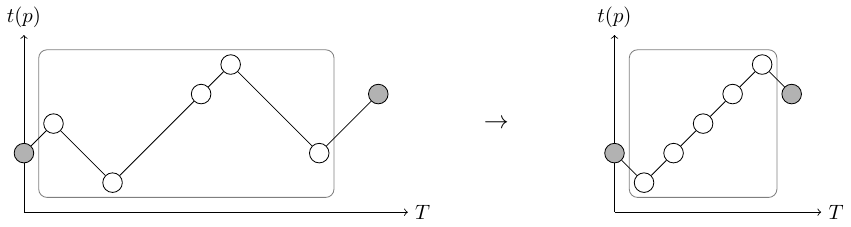}
        \captionsetup{width=0.8\textwidth}
    \caption{Sketch of improvement of a schedule by sorting the blocks internally.}
    \label{fig:enter-label}
    \end{center}
\end{figure}

\begin{proof} ~
\begin{enumerate}[label=\alph*)]
    \item Let $\block = (\job_1,\ldots,\job_n)$ with $\tempJ_i := \tempJ(\job_i)$ and assume $\TminB = \tempJ_r$ and $\TmaxB = \tempJ_s$ with $r < s$. Then 
    \begin{align*}
    \sum_{i=1}^{n-1} |\tempJ_{i}-\tempJ_{i+1}| &= \sum_{i=1}^{r-1} |\tempJ_{i}-\tempJ_{i+1}| + \sum_{i=r}^{s-1} |\tempJ_{i}-\tempJ_{i+1}| + \sum_{i=s}^{n-1} |\tempJ_{i}-\tempJ_{i+1}| \\
    %&\geq \sum_{i=1}^{r-1} |\tempJ_{i}-\tempJ_{i+1}| + \sum_{i=r}^{s-1} (\tempJ_{i}-\tempJ_{i+1}) + \sum_{i=s}^{n-1} |\tempJ_{i}-\tempJ_{i+1}| \\
    &\geq \sum_{i=1}^{r-1} |\tempJ_{i}-\tempJ_{i+1}| + (\tempJ_r - \tempJ_s) + \sum_{i=s}^{n-1} |\tempJ_{i}-\tempJ_{i+1}| \\
    %&\geq (\tempJ_r - \tempJ_s) = 
    &\geq (\TmaxB - \TminB).
    \end{align*}
    Now let $\block$ be sorted, \wlg decreasingly, then $\temp(\block) = \sum_{i=1}^{n-1} |\tempJ_{i}-\tempJ_{i+1}|  = \tempJ_1 - \tempJ_n = \TmaxB- \TminB$.

\item For schedules consisting of only one block the claim follows by a), therefore let $\sched = (\block_1,\ldots,\block_{k+1})$ with $k \geq 1$ and let $i \in \{1,\ldots,k+1\}$. Let $\job, \job'$ be the first and last job of $\block_i$, respectively, and let $\tilde{\job}$ be the last job of $\block_{i-1}$ and $\tilde{\job}'$ the first job of $\block_{i+1}$, where we set $\job := \tilde{\job}$ in the case $i=1$ and $\job' := \tilde{\job}'$ in the case $i=k$. Then it holds 
\begin{equation} \label{eq:lem1-1}
\temp(\sched) = \temp(\block_1,\ldots,\block_{i-1}) + \temp(\tilde{\job}, \job) + \temp(\block_i) + \temp(\job',\tilde{\job}') + \temp(\block_{i+1},\ldots,\block_k).
\end{equation}

Let $\job_{i_r}$ and $\job_{i_s}$ be the jobs in $\block_i$ with minimal and maximal process temperature, respectively. W.l.o.g.\@\xspace assume $r < s$, otherwise reverse the schedule. Then it holds
\begin{align}
    \temp(\block_i) &= \temp(\job, \ldots, \job_{i_{r}}) + \temp(\job_{i_r}, \ldots, \job_{i_{s}}) + \temp(\job_{i_s}, \ldots, \job') \\
    &\geq^{a)} \temp(\job, \ldots, \job_{i_{r}}) + \temp(\overrightarrow{\block}_i) + \temp(\job_{i_s}, \ldots, \job') \label{ineq:lem1-1}
\end{align}

By inductively applying the triangle inequality, we see that $\temp(\tilde{\job},\job) + \temp(\job, \ldots, \job_{i_{r}}) \geq \temp(\tilde{\job}, \job_{i_r})$ and $\temp(\job_{i_s}, \ldots, \job') + \temp(\job', \tilde{\job}') \geq \temp(\job_{i_s}, \tilde{\job}')$. By inserting (\ref{ineq:lem1-1}) into (\ref{eq:lem1-1}) and applying those inequalities, it follows

\begin{equation}
\begin{aligned}
    \temp(\sched) &\geq \temp(\block_1,\ldots,\block_{i-1}) + \temp(\tilde{\job}, \job_{i_r})+ \temp(\overrightarrow{\block}_i) + \temp(\job_{i_s},\tilde{\job}') \\ 
    &\quad + \temp(\block_{i+1},\ldots,\block_k) \\
    &= \temp(\block_1,\ldots,\block_{i-1},\overrightarrow{\block}_i,\block_{i+1},\ldots,\block_k) \\
\end{aligned}
\end{equation}

Applying this procedure to every block in $\sched$ results in a schedule $\sched'$ with internally sorted blocks, whose total temperature change is bounded above by $\temp(\sched)$.
\end{enumerate}

\end{proof}

\begin{lemma}\label{lem:nointersectionblocks} ~
Any schedule $\sched$ can be improved to $\sched' = (\block'_1,\ldots,\block'_{k+1})$ such that 
\begin{enumerate}[leftmargin=1.1cm, label=\roman*)]
    \item \label{lem:nointersectionblocksi} Each block $\block'_i$ is internally sorted.
    \item \label{lem:nointersectionblocksii} For all $r < s \in \{1,\ldots,k+1\}$ both either even or odd it holds either $\TmaxB[r]['] < \TminB[s][']$ or $\TmaxB[s]['] < \TminB[r][']$, \ie different blocks in $\sched'$ of the same color do not intersect \wrt the process temperature of their jobs.
\end{enumerate}
\end{lemma}

\begin{proof} Any schedule can be improved to $\sched = (\block_1,\ldots,\block_{k+1})$ that fulfills \ref{lem:nointersectionblocksi} by \Cref{lem:internalorderingblock}. Let some $r<s$ both even or odd be given with \wlg $\TmaxB[r] \geq \TminB[s]$ and $\TminB[r] \leq \TmaxB[s]$ (note that actually strict inequality holds due to uniqueness of process temperature values). Furthermore we assume $\TminB[r] < \TminB[s]$, the other cases of intersection follow analogously.

We make two simple observations: First, removing any job from $\sched$ results in a sequence of jobs with both $\temp$ and $\col$ bound above by $\temp(\sched)$ and $\col(\sched)$, respectively. Conversely, plugging a job $\job$ into any sequence of jobs between $\job'$ and $\job''$ such that $\colJ(\job') = \colJ(\job)$ or $\colJ(\job) = \colJ(\job'')$ and either $\tempJ(\job')< \tempJ(\job) < \tempJ(\job'')$ or $\tempJ(\job') > \tempJ(\job) > \tempJ(\job'')$ also does neither increase $\temp$ nor $\col$.

\begin{figure}[ht!]
    \begin{center}
        \includegraphics[width=0.8\textwidth]{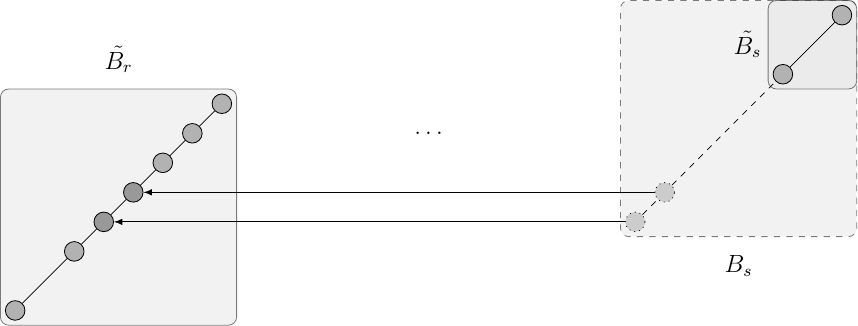}
                \captionsetup{width=0.8\textwidth}
    \caption{Illustration of an improvement by removing intersections. Jobs that fit properly into another block of the same color can be inserted without increasing both $\temp$ or $\col$.}
    \label{fig:blocksnointersect}
    \end{center}
\end{figure}

Hence by removing the job with minimal process temperature from $\block_s$ and plugging it into the unique fitting position in $\block_r$ (as $B_r$ is internally sorted) does neither increase $\temp$ not $\col$. This is due to the fact that the border items of $\block_r$ do not change, since $\TminB[r] < \TminB[s] < \TmaxB[r]$ by assumption. Doing so inductively results in new blocks $\tilde{\block_r}$ and $\tilde{\block_s}$ with no intersection (note that $\tilde{\block_s}$ might be empty). See \Cref{fig:blocksnointersect} for a visualization.

Now, in the new schedule, the number of pairs of same colored blocks that intersect have strictly decreased by at least one. Proceeding inductively over those pairs results in a schedule $\sched'$ that fulfills conditions \ref{lem:nointersectionblocksi} and \ref{lem:nointersectionblocksii} with both $\temp$ and $\col$ bound above by the values of $\sched$. 
\end{proof}

For the next properties we need a technical lemma; the proof is contained in \Cref{sec:appendix}.
\begin{customlemma}{\ref{lem:helperequations}} %[Hilfslemma]
Let $a,b,c,d \in \mathbb{N}$ with $b < c$ and $a < d$. Then
\begin{align*}
       |b-a| + |d-c| \leq  |c-a| + |d-b|.
\end{align*}
\end{customlemma}

\ref{lem:helperequations} can be used to determine the sorting of a block, depending on its surrounding blocks: Let $\block_{i-1},\block_i,\block_{i+1}$ be a subsequence of some schedule $\sched$ with properties \ref{lem:nointersectionblocksi} and \ref{lem:nointersectionblocksii}. Let $p$ be the last job of block $\block_{i-1}$ and $p'$ the first one of block $\block_{i+1}$. Then, if $t(p) < t(p')$, we can \wlg sort $\block_i$ increasingly, without increasing $\temp$; analogously a decreasing sorting is suitable if $t(p) > t(p')$. Both can be seen by choosing $a:= t(p), d := t(p'), b:= \TminB[i]$ and $c:= \TmaxB[i]$. 

\begin{figure}[ht!]
    \begin{center}
        \includegraphics[width=0.5\textwidth]{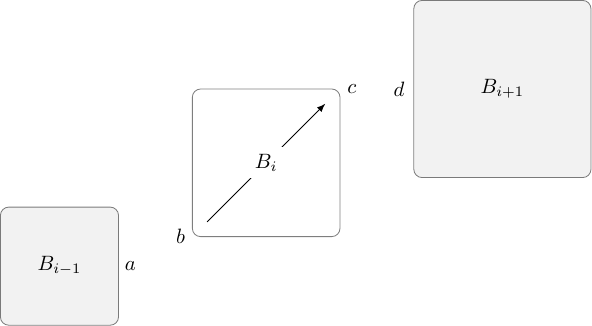}
        \captionsetup{width=0.8\textwidth}
    \caption{Illustration of the internal sorting of blocks, depending on the surrounding blocks.}
    \label{fig:innerblocksordering}
    \end{center}
\end{figure}

Additionally this can be determined without knowledge about the sorting of $\block_{i-1}$ and $\block_{i+2}$; as they don't intersect it holds $t(p) < t(p')$
if and only if $\TmaxB[i-1] < \TmaxB[i+1]$, equivalently for the minima or by comparison of the process temperatures of every pair of jobs from $\block_{i-1}$ and $\block_{i+1}$.

\begin{lemma}\label{lem:blocksorderedexternally}
Any schedule $\sched$ can be improved to $\sched' = (\block'_1,\ldots,\block'_{k+1})$ such that 
\begin{enumerate}[leftmargin=1.1cm, label=\roman*)]
    \item \label{lem:blocksorderedexternallyi} Each block is internally sorted.
    \item \label{lem:blocksorderedexternallyii} Different blocks of the same color do not intersect.
    \item \label{lem:blocksorderedexternallyiii} For all $r < s \in \{1,\ldots,k+1\}$ both either even or odd it holds $\TmaxB[r]['] < \TminB[s][']$, \ie the two subsequences of blocks of the same color are strictly increasing externally.
\end{enumerate}
\end{lemma}

\begin{proof} Any schedule can be improved to $\sched$ with \ref{lem:blocksorderedexternallyi} and \ref{lem:blocksorderedexternallyii} by \Cref{lem:nointersectionblocks}. Assume $\sched$ contains more than $3$ blocks and that the schedule is not externally sorted in decreasing order, both cases can be solved by possibly reversing $\sched$. 

First we observe that $\sched$ fulfills \ref{lem:blocksorderedexternallyiii} as well if and only if there is no sequence $B_{i-1},B_i,B_{i+1},B_{i+2}$ such that $(B_{i-1},B_{i+1})$ does not have the same external sorting as $(B_i, B_{i+2})$, i.e. that either $\TmaxB[i-1] < \TmaxB[i+1]$ and $\TmaxB[i] > \TmaxB[i+2]$ or $\TmaxB[i-1] > \TmaxB[i+1]$ and $\TmaxB[i] < \TmaxB[i+2]$. Clearly if there is such a sequence, then \ref{lem:blocksorderedexternallyiii} can not be fulfilled. Assume \ref{lem:blocksorderedexternallyiii} does not hold. If the two subsequences of same-colored blocks are externally sorted contrariwise, then the case is clear. Otherwise, since $\sched$ is not externally sorted in decreasing order, there is a sequence $B_{j-1}, B_{j+1}, B_{j+3}$ with $\TmaxB[j-1] < \TmaxB[j+1] > \TmaxB[j+3]$ or $\TmaxB[j-1] > \TmaxB[j+1] < \TmaxB[j+3]$ for some $j$. Assume the first case, then if $\TmaxB[j] < \TmaxB[j+2]$, then $\block_j,\block_{j+1},\block_{j+2},\block_{j+3}$ is the sequence searched for. Otherwise, \ie in the case $\TmaxB[j] > \TmaxB[j+2]$, its $\block_{j-1},\block_{j},\block_{j+1},\block_{j+2}$. The case $\TmaxB[j-1] > \TmaxB[j+1] < \TmaxB[j+3]$ works fully analogous, which proves the equivalence.

\begin{figure}[ht!]
    \begin{center}
        \includegraphics[width=\textwidth]{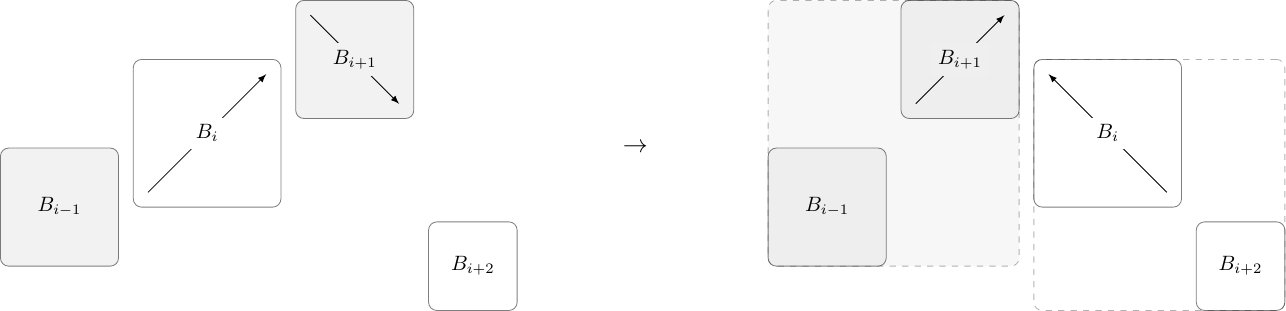}
            \captionsetup{width=0.8\textwidth}
    \caption{Improvement of a schedule by swapping $\block_i$ and $\block_{i+1}$ in a subsequence with externally contrariwise sorted pairs of blocks.}
    \label{fig:blocksorderedexternally}
    \end{center}
\end{figure}

We show now that in case, that such a sequence exists, the schedule can be improved as visualized in \Cref{fig:blocksorderedexternally}: Given such a sequence, the idea is to swap $\block_i$ and $\block_{i+1}$. Without restriction we assume again that $\TmaxB[i-1] < \TmaxB[i+1]$ and $\TmaxB[i] > \TmaxB[i+2]$ holds, the other case works analogously. We swap the blocks $\block_{i+1}$ and $\block_i$ and create a new schedule $\tilde{\sched}$ by setting 
    \[
     \tilde{\sched} := (\block_1, \ldots, \block_{i-1}, \overrightarrow{\block_{i+1}}, \overleftarrow{\block_{i}},\block_{i+2},\ldots, \block_{k+1})
    \]
    By \Cref{lem:helperequations} we know that $\block_i$ is \wlg internally sorted in increasing order in $\sched$ and $\block_{i+1}$ decreasingly. Let $p$ be the last job in $\block_{i-1}$ and $p'$ the first one of $\block_{i+2}$. Then the difference $\temp(\sched) - \temp(\tilde{\sched})$ computes as
    \begin{align*}
        &|\TminB[i] - t(p)| + |\TmaxB[i+1]-\TmaxB[i]| + |\TminB[i+1]-t(p')| \\ 
        &-(|\TminB[i+1] - t(p)| + |\TmaxB[i+1]-\TmaxB[i]| + |\TminB[i]-t(p')|)\\
        =&|\TminB[i] - t(p)| + |\TminB[i+1]-t(p')| \\
        &-(|\underbrace{\TminB[i+1] - t(p)}_{>0}| +|\underbrace{\TminB[i]-t(p')}_{>0}|) \\
        =&|\TminB[i] - t(p)| + |\TminB[i+1]-t(p')| \\
        &-(|\TminB[i+1] - t(p) + \TminB[i]-t(p')|)  \\
        \geq&|\TminB[i] - t(p)| + |\TminB[i+1]-t(p')| \\
        &-(|\TminB[i+1] -t(p')| + |\TminB[i]-t(p)|)  = 0.
    \end{align*}

Note that $\block_{i-1}$ and $ \overrightarrow{\block_{i+1}}$ have the same color, as well as $\overleftarrow{\block_{i}}$ and $\block_{i+2}$. Analogously to the proof of \Cref{lem:internalorderingblock} b) one can show, that there are internal sortings $\block$ of $(\block_{i-1}, \overrightarrow{\block_{i+1}})$ and $\block'$ of $(\overleftarrow{\block_{i}}, \block_{i+2})$, such that \[\temp(\block_1,\ldots,\block_{i-2},\block,\block',\block_{i+3},\ldots,\block_{k+1}) \leq \temp(\tilde{\sched}).\] The number of blocks have decreased by $2$, furthermore this schedule clearly fulfills \ref{lem:blocksorderedexternallyi} and \ref{lem:blocksorderedexternallyii}. Hence proceeding inductively over those quadruplets and possibly reversing the schedule afterwards results in a schedule $\sched'$ fulfilling \ref{lem:blocksorderedexternallyi} to \ref{lem:blocksorderedexternallyiii} as desired.
\end{proof}

Now we are able to prove \Cref{thm:mainthm}:

\begin{customthm}{\ref{thm:mainthm}} 
There exists an optimal schedule $\sched^* = (B_1,\ldots,B_{k+1})$ with $k \leq \hat{k}$ such that:
\begin{enumerate}[leftmargin=1.1cm, label=\roman*)]
    \item \label{thm:mainthmi}Each block is internally sorted.
    \item \label{thm:mainthmii} Blocks of the same color do not intersect.
    \item \label{thm:mainthmiii} Both subsequences of same-colored blocks are externally sorted in increasing order.
    \item \label{thm:mainthmiv} The blocks $B_2, \ldots, \block_k$ are internally sorted in increasing order.
\end{enumerate}
\end{customthm}

\begin{proof}
Let an optimal schedule $S$ of at most $\hat{k}$ color changes be given, then by \Cref{lem:blocksorderedexternally} we can construct an improvement $S'$ fulfilling \ref{thm:mainthmi} to \ref{thm:mainthmiii}. The claim follows now by \Cref{lem:helperequations}; for any internal block $\block'_i$ of $\sched'$ we know that $\TmaxB[i-1]['] < \TmaxB[i+1][']$ holds by \ref{thm:mainthmiii}, hence \Cref{lem:helperequations} implies that $\block'_i$ is \wlg internally sorted in increasing order. The total process temperature change of the resulting schedule $\sched^*$ agrees with $\temp(\sched)$ due to optimality, proving the claim.
\end{proof}

In the upcoming section we introduce the algorithm to solve $(\textbf{P})$. Given sets $\{\block_i^0\}, \{\block_j^1\}$ of pairwise disjoint, non-intersecting blocks of colors $0$ and $1$, respectively, with $(\cup_i \block_i^0) \cup (\cup_j \block_j^1) = \setJ$, then \Cref{thm:mainthm} determines the external ordering of the blocks as well as the internal sorting of the inner blocks of an optimal schedule. 

But the internal sorting of the border blocks is known as well: Since the same-colored blocks increase externally, the first block contains the first $n_1$ jobs $\{p_1,\ldots,p_{n_1}\}$ of color $c$, sorted by process temperature, where $n_1 = |B_1|$. Analogously block $B_{k+1}$ contains the last $n_{k+1}$ jobs of its color. Let $p$ be the first job of $B_2$ and $p_1,p_{n_1}$, the first and last job of $B_1$, respectively, then it follows that $B_1$ is \wlg sorted increasingly, if $|\tempJ(p)-\tempJ(p_{n_1})| \leq |\tempJ(p)-\tempJ(p_{1})|$, and in decreasing order otherwise. Similarly the last block $B_{k+1}$ is ordered depending on the process temperature difference from its first and last job to the last the job of maximal process temperature of $B_k$. Thus no matter the internal sorting of $\block_1$, the transition costs from $B_1$ to $B_2$ is $\min(|\tempJ(p)-\tempJ(p_{n_1})|, |\tempJ(p)-\tempJ(p_{1})|)$. Similar considerations hold for the internal sorting of the last block, which determines the complete schedule. In the following section we use this together with \Cref{thm:mainthm} to define an algorithm for construction of the blocks.

\begin{figure}[ht!]
    \begin{center}
        \includegraphics[width=0.65\textwidth]{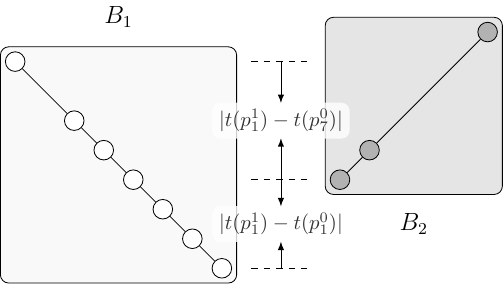}
      \captionsetup{width=0.8\textwidth}  
    \caption{The internal sorting of the border blocks depend on the process temperature differences of the first job in $\block_1$, which is the seventh job $\job^0_7$ of color $0$ \wrt process temperature, and $\job^0_1$ to $\job_1^1$, as the latter is the first job of $\block_2$ by \Cref{thm:mainthm} and $\block_1$ is sorted internally.}
    \label{fig:firstblock}
    \end{center}
\end{figure}

\section{A Shortest Path Algorithm} \label{sec:algorithm}
    In the following, we introduce an algorithm to solve (\textbf{P}), based on the structural results established in the previous section. The algorithm performs a shortest path search in a directed, weighted graph whose paths encode feasible schedules satisfying the properties of \Cref{thm:mainthm}. The graph has polynomial size in the number of jobs and colors, which establishes the tractability of (\textbf{P}) by the use of Dijkstra's algorithm.

Let \(N_0\) and \(N_1\) denote the number of jobs of color $0$ and $1$, respectively. We label the jobs of color \(c \in \{0,1\}\) by \(\{\job_1^c, \ldots, \job_{N_c}^c\}\), sorted increasingly by their process temperature, i.e., \(\tempJ(\job_i^c) < \tempJ(\job_{i+1}^c)\) for all \(i = 1, \ldots, N_c - 1\). For improved readability, we refer to color $0$ as \emph{white} and color $1$ as \emph{black}. For technical reasons we assume that there are at least three jobs and both colors represented with $\hat{k} \geq 2$.

The graph consists of $\hat{k}-1$ \emph{layers}, each composed of two \emph{cards} - a white card and a black card - corresponding to the two colors. Each card contains an \(N_0 \times N_1\) grid of nodes and is traversed either vertically (white card) or horizontally (black card). Rows correspond to white jobs, and columns to black jobs, both sorted by increasing process temperature. The edge weights represent temperature differences between consecutive jobs. For example, an edge from row \(i\) to \(i+1\) on a white card has weight \(|\tempJ(\job_{i+1}^0) - \tempJ(\job_i^0)|\). Analogously, edges between columns \(j\) and \(j+1\) on a black card have weight \(|\tempJ(\job_{j+1}^1) - \tempJ(\job_j^1)|\). \Cref{fig:cards-algorithm} illustrates the structure the cards.

\begin{figure}[ht!]
\begin{center}
\includegraphics[width=0.7\textwidth]{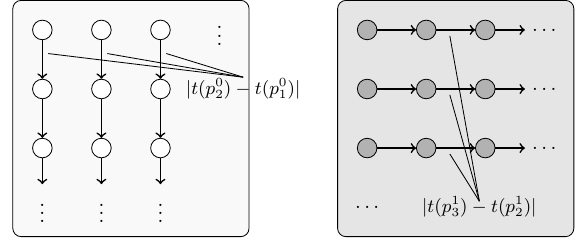}
\captionsetup{width=0.8\textwidth}
\caption{A pair of cards of a layer.} \label{fig:cards-algorithm}
\end{center}
\end{figure}

The first layer (layer 1) is connected to a special construction (layer 0) that models the initial block, which can be ordered in either increasing or decreasing temperature. Layer 0 contains a row of white jobs and a column of black jobs, with traversal costs defined as above.
Each node in the entry graph of layer 0 is connected to the corresponding card of the opposite color, with edge weights reflecting the transition cost between blocks, as discussed at the end of \Cref{sec:theoreticalresults}. The graph also includes a start node \(s\), which connects to the lowest-temperature job of each color (\(\job_1^0\) and \(\job_1^1\)).

Every layer contains an \emph{exit graph}, representing the final block. It includes a row or column of nodes for jobs of each color and a final node \(t_l\), where \(l\) is the index of the layer. Nodes in the exit graph are connected by edges from the opposite-colored card of the same layer, using minimum transition costs depending on the sorting of the last block. Finally, the nodes \(t_l\) are connected to the global target node \(t\). \Cref{fig:FirstLayer-algorithm} shows an example of layer $1$ with entry and exit graphs.

\begin{figure}[ht!]
\begin{center}
    \includegraphics[width=0.8\textwidth]{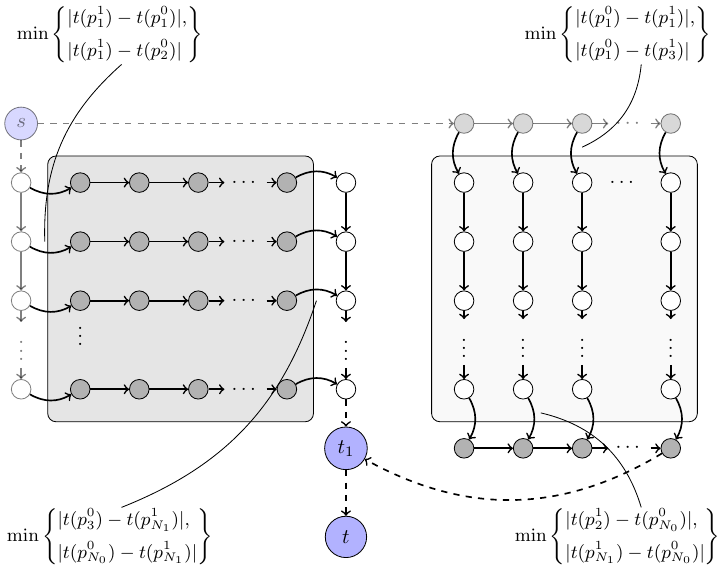}
\captionsetup{width=0.8\textwidth}
 \caption{Exemplary sketch of the first layer with entry and exit graph. The exit graph exists for each layer, the dashed edges have a weight of $0$.}
 \label{fig:FirstLayer-algorithm}
 \end{center}
\end{figure}

Transitions between layers occur only via the cards. Each white card connects to the black card in the next layer, and vice versa. For a white card, we add edges from \((i,j)\) to \((i+1,j)\) for \(i = 1,\ldots,N_0 - 1\), \(j = 1,\ldots,N_1\). For black cards, we add edges from \((i,j)\) to \((i,j+1)\) for \(i = 1,\ldots,N_0\), \(j = 1,\ldots,N_1 - 1\). These edges model color transitions and have weight \(|\tempJ(\job_j^1) - \tempJ(\job_i^0)|\). \Cref{fig:Cardconnection-algorithm} illustrates this structure.

\begin{figure}[ht!]
\begin{center}
\includegraphics[width=0.8\textwidth]{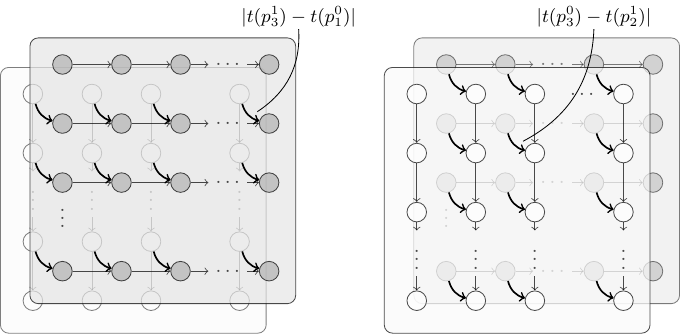}
\captionsetup{width=0.8\textwidth}
\caption{Connection of the layers by the cards.}
\label{fig:Cardconnection-algorithm}
\end{center}
\end{figure}
Note that while the figures display the full cards in each layer for clarity, the actual number of nodes in the search graph may be reduced, as in each subsequent layer one additional row or column is excluded from connectivity.

It is easy to see - without a formal proof - that each \(s - t_l\) path corresponds to a schedule \(\sched\) with exactly \(l-1\) color changes, such that the total path length equals \(\temp(\sched)\). Each job index appears exactly once, and edges enforce increasing temperature order within blocks. The properties of \Cref{thm:mainthm} are thus satisfied by construction. Furthermore, for every schedule \(\sched\) fulfilling those properties, there exists a corresponding \(s - t_k\) path, where \(k+1\) is the number of blocks. Hence, by finding a shortest path from \(s\) to \(t\), we obtain a schedule with at most \(\hat{k}\) color changes and minimal total temperature cost.

\begin{corollary}
The problem (\textbf{P}) to find a schedule of at most $\hat{k}$ color changes and minimal total process temperature change is solvable in polynomial time.
\end{corollary}

\section{Outlook} \label{sec:outlook}
    In this paper, we provided a formal description of the calendering problem in the context of flooring production, where each job is associated with a temperature and a color. By exploiting structural properties of optimal schedules, we were able to construct a tailored graph-based model that solves the bicriteria problem in polynomial time when the number of colors is two. As a natural direction for further research, one may consider the case of an arbitrary number of colors. In this setting, the total number of color changes is measured nominally, that is,
\[
\col(\job_1,\ldots,\job_N) := \left|\left\{i \in \{1,\ldots,\numJ - 1\} \mid \colJ(\job_i) \neq \colJ(\job_{i+1})\right\}\right|.
\]
The complexity of the general problem remains open to our knowledge. A second, intermediate extension of interest is the case where the number of colors is fixed but greater than two. The structural results established in \Cref{sec:theoreticalresults} do not carry over in a straightforward manner. While \Cref{lem:internalorderingblock} and \Cref{lem:nointersectionblocks} remain valid for any number of colors, the internal structure of color blocks may deviate significantly from the pattern seen in the two-color case. To illustrate this, consider an instance with three colors $\{0,1,2\}$ and the set of jobs
\[
\setJ = \{(1,0), (2,0),(3,0),(4,0), (1,1),(3,1),(0,2),(2,2),(4,2),  (5,2)\},
\]
with maximum allowed color variation $\hat{k} = 4$. The schedule
\[
\sched = (\underbrace{(0,2), (2,2)}_{\block_1}, \underbrace{(2,0),  (1,0)}_{\block_2}, \underbrace{(1,1), (3,1)}_{\block_3}, \underbrace{(3,0), (4,0)}_{\block_4}, \underbrace{(4,2), (5,2)}_{\block_5})
\]
as depicted in \Cref{fig:counterexample} is the unique optimal schedule (up to reversion) of minimum total process temperature change $\temp(\sched) = 7$, verified via full enumeration. Notably, blocks $\block_2$ and $\block_4$, both inner blocks of color $0$, do not share the same internal ordering. Hence, \Cref{thm:mainthm} does not apply directly, and additional structural understanding will be required to extend our approach to three or more colors.

\begin{figure}[ht!]
\begin{center}
\includegraphics[width=0.6\textwidth]{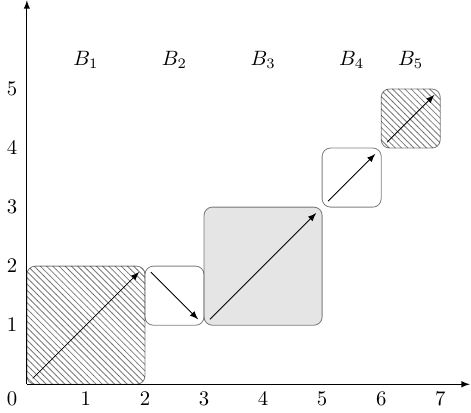}
\captionsetup{width=0.8\textwidth}
\caption{An example for an instance with three colors, where blocks of the same color do not share the same internal ordering, which holds for all optimal schedules.}
\label{fig:counterexample}
\end{center}
\end{figure}

Overall, these extensions open up interesting avenues for future work, both from a combinatorial and algorithmic point of view.

\section*{Declaration of competing interest}

The authors declare that they have no known competing financial interests or personal relationships that could have appeared to influence the work reported in this paper.

\printbibliography

\appendix
\section{Appendix} \label{sec:appendix}
    \begin{lemma}\label{lem:helperequations} %[Hilfslemma]
Let $a,c,b,d \in \mathbb{N}$ with  $b < c$ and $a < d$. Then
\begin{align}
        |b-a| + |d-c| \leq |c-a| + |d-b| . \label{eq:helperequation}
\end{align}

\end{lemma}

\begin{proof} ~
We show this by case distinction. First assume $b \geq d$ then $a < d  \leq b < c$, hence
\begin{align*}
    |c-a| + |d-b| - (|b-a| + |d-c|) &= (c-a) - (b-a) + (b-d) - (c-d) \\
    &= (c-b) + (b-c)= 0.
\end{align*}
Analogously if  $a \geq c$ we have $b < c \leq a < d$ and also
\begin{align*}
    |c-a| + |d-b| - (|b-a| + |d-c|) &= (a-c) - (d-c) + (d-b) - (a-b) \\
    &= (a-d) + (d-a) = 0.
\end{align*}

The last case is $a < c$ and $b < d$. For this we consider $4$ sub cases, starting with $b < a$ and $c\leq d$, then $b< a < c\leq d$,
\begin{align*}
    |c-a| + |d-b| - (|b-a| + |d-c|) &=  (d-b)-(a-b) +  (c-a)  - (d-c) \\
    &= 2(c-a) > 0.
\end{align*}

Next $b \geq a$ and $c\leq d$, then $a \leq b < c \leq d$, leading to
\begin{align*}
    |c-a| + |d-b| - (|b-a| + |d-c|) &=  (d-b)-(b-a) +  (c-a)  - (d-c) \\
    &= 2(c-b) > 0.
\end{align*}
Next $b < a$ and $c\geq d$, then $b < a < d \leq c$, leading to
\begin{align*}
    |c-a| + |d-b| - (|b-a| + |d-c|) &=  (d-b)-(a-b) +  (c-a)  - (c-d) \\
    &= 2(d-a) > 0.
\end{align*}
Lastly $b \geq a$ and $c\geq d$, then $a \leq b < d \leq c$, leading to
\begin{align*}
    |c-a| + |d-b| - (|b-a| + |d-c|) &=  (d-b)-(b-a) +  (c-a)  - (c-d) \\
    &= 2(d-b) > 0.
\end{align*}

Altogether, \Cref{eq:helperequation} holds in all cases.
\end{proof}

\section*{Glossary} \label{sec:glossary}
    
\small 
\renewcommand{\arraystretch}{1.5}
\setlength{\LTpre}{0pt} 
\setlength{\LTpost}{0pt} 
\begin{longtable}[l]{@{}lp{0.8\textwidth}@{}}
$\setJ$ & set of jobs/products \\
$\job$ & a job, which is a pair $(\tempJ,\colJ)$ of process temperature $\tempJ$ and color $\colJ$ \\
$\numJ$ & number of jobs \\
$\numJ_0$ & number of jobs of color $0$ \\
$\numJ_1$ & number of jobs of color $1$ \\
$\sched$ & a schedule of a set of jobs $\setJ$ \\
$\tempJ(p)$ & process temperature of $\job$ \\
$\colJ(p)$ & color of $\job$ \\
$\temp$ & total process temperature change of sequence/block/schedule of jobs \\
$\col$ & total color change of sequence/block/schedule of jobs \\
$\TminB$ & minimal process temperature value of all jobs in $\block$ \\
$\TmaxB$ & maximal process temperature value of all jobs in $\block$ \\
$\block$ & a block, a sequence of jobs of the same color \\
$\overrightarrow{\block}$ & block of jobs of $\block$, sorted increasingly \wrt $\tempJ$ \\
$\overleftarrow{\block}$ & block of jobs of $\block$, sorted decreasingly \wrt $\tempJ$ \\
$\overline{\block}$ & reversion of $\block$ 
\end{longtable}
\end{document}